\documentclass{svmult}

\usepackage{amsmath}
\usepackage{amssymb}
\usepackage{mathrsfs}
\usepackage{enumerate}

\numberwithin{equation}{section}

\newcommand{\real}{\mathbb{R}}

\newcommand{\Oun}{\mathcal{O}(1)}

\title*{On the Complexity of some Geometrical Objects}
\newcommand{\programme}{\mathtt{P}}
\newcommand{\ensprog}{\mathscr{P}}
\newcommand{\out}{\mathtt{out}}
\newcommand{\ind}{\mathscr{I}}
\newcommand{\indb}{\{0,1\}}
\newcommand{\entier}{\mathbb{N}}
\newcommand{\cantor}{\mathscr{C}}
\smartqed

\author{P. COLLET}
\institute{P.Collet \at Centre de Physique Th\'eorique,
CNRS UMR 7644,
Ecole Polytechnique,
F-91128 Palaiseau Cedex (France).
\email{collet@cpht.polytechnique.fr}}

\begin{document}

\maketitle

\abstract{
We recall the definition of the $\epsilon$-distortion
  complexity of a 
set defined in \cite{bcc} and the results obtained in this paper for
Cantor sets of the interval defined by iterated function systems. We 
state an analogous  definition for measures which may be more useful
when dealing with dynamical systems. We prove a new lower bound in the
case of Cantor sets of the interval defined by analytic 
iterated function systems. We also give an upper bound the
$\epsilon$-distortion complexity of invariant sets of uniformly
hyperbolic dynamical systems.}

\section{Introduction.}

It is common sense that some abjects are simple like lines, circles,
planes etc.  In Mathematics
and Physical sciences one also finds objects of different nature like
fractal sets which are undoubtedly of more complicated nature. One can
ask if it is possible to describe quantitatively this difference of
complexity, namely can one define a number measuring the complexity of say
a geometrical object. There are certainly many possible definitions.
In section \ref{sec1} we describe such a possible definition (see
\cite{bcc}).
 
A natural question is then to ask if this quantity is related to other
Mathematical quantities measuring different properties of the object
(dimension for example).

We now explain briefly some of the ideas behind our definition. In
his seminal paper on information \cite{K0}, Kolmogorov gave three
possible definition for the quantity of information contained in a
sequence of zeros and ones. They can also be
viewed as a quantitative approach to the measure of
complexity of such sequences. 
We will refer below to two  definitions of Kolmogorov
which can be briefly summarized as follows. 

In one definition
(algorithmic complexity, see for example \cite{lv}),
  a (finite) sequence is said to be
complex if there does not exist a  short computer program that
generates this object. Quantitatively, the complexity is measured by
the length of the smallest computer program that outputs this
sequence. One can also use the total number of instructions executed
by the machine to produce the sequence 
which may be much larger that the length of the program (due to loops
for example). 

In another definition of Kolmogorov (the combinatorial complexity, see
for example \cite{rsv}),
the idea is that an object is complex if it is contained in a large
set. Picking a particular object in  a large set of equivalent objects
require a large information. The opposite situation is even more
obvious, if a set contains only one object, it is easy to pickup this
object.

The definitions of Kolmogorov are for finite sequences of zeros and
ones but we want to analyze continuous objects. One needs therefore
some kind of discretization (see \cite{asarin} for similar questions). 
One can think for example of drawing the object on a computer
screen. In this case, very fine structures of the object  are somehow
irrelevant if their size is smaller than a pixel size. This leads
naturally to the ideas that  we have a fixed given precision
$\epsilon>0$, and the object will be described by a finite number of
points covering the object up a to precision $\epsilon$. More
precisely, we will impose that the  Hausdorff distance (see
\cite{munkres}) between the object and
the finite set of points used for its description is smaller than $\epsilon$.

Summarizing, we will start with a given reference frame with units,
and look for finite set of points with rational coordinates whose  
Hausdorff distance to the object is smaller than $\epsilon$. Finally
we consider the programs generating the
coordinates of these points. We finally optimize on the points
positions and the program length. The optimal program length is what
we call the distortion complexity of the object at precision $\epsilon$.

We will also propose a similar definition for the complexity of measures
which may be more adapted to the study of dynamical systems.

We refer to the book \cite{bi} which develops similar ideas with a
different goal. See also \cite{CS}, \cite{braverman}, 
\cite{braverman2}, \cite{brudno}, \cite{valentin} and references therein
for related works.

\section{Definitions and main results.}\label{sec1}

We first recall some definitions related to the Kolmogorov complexity (see
\cite{K0} and \cite{lv}). We will denote by $\ensprog$ the set of
finite sequences of zeros and ones. We will denote by $l(\programme)$ the
length of the sequence $\programme$. An element $\programme\in
\ensprog$ will be considered as a program working on a computer
(universal Turing machine). 

We consider in $\real^{n}$ a fixed orthogonal basis and a unit of
length. We will consider below the programs whose output is a finite
set of points in $\real^{n}$ described by their coordinates (these are
programs which terminate).  For a program $\programme\in \ensprog$ we
will denote by $\out(\programme)$ this finite set of points.

We also recall that the Hausdorff distance between two closed sets
$F$ and $F'$ is defined by
$$
d\big(F,F'\big)=
\max\big\{\sup_{x\in F}d(x,F')\,,\,\sup_{y\in F'}d(y,F)\big\}\;.
$$
This is a distance between closed sets (see \cite{munkres} for more
properties) which measures how the two sets differ.

We can now formulate our main definition, the idea is that given a
precision $\epsilon$, we look for the approximation of a set $F$ by
finitely many points which
requires the smallest computer program.

\begin{definition}\label{disto}
For  $\epsilon>0$, the $\epsilon$-distortion complexity
of a closed set $F$ (denoted by $\Delta_{\epsilon}(F)$)  is the number
$$
\Delta_{\epsilon}(F)=\min\big\{l(\programme)\,
\big|\,\programme\in\ensprog,\; 
d\big(F,\out(\programme)\big)\le \epsilon
\big\}
$$ 
\end{definition}

Our goal in the sequel is to understand how $\Delta_{\epsilon}(F)$
depends on $\epsilon$ for $\epsilon$ small and how it can be related
to other quantitative properties of $F$, at least for some particular
classes of sets. When there is no ambiguity on
the set $F$, we will use the notation $\Delta_{\epsilon}$ instead
of $\Delta_{\epsilon}(F)$.

There is an easy upper bound for $\epsilon$
distortion complexity of a set $F$ in terms of its  box counting
dimension $d_{F}$ (see \cite{falconer}). 
Given $\epsilon>0$, we can cover the set by at most
$\epsilon^{-d_{F}}$ balls of radius $\epsilon$. We can use the centers
of these balls to describe $F$ at precision $\epsilon$. To describe 
a point in a finite dimensional space we can give the dyadic expansion
of its coordinates. This immediately leads to a bound
\begin{equation}\label{bornetriviale}
\Delta_{\epsilon}(F)\le 
\mathcal{O}(1) \; \epsilon^{-d_{F}}\;\log\epsilon^{-1}\;.
\end{equation}
We will see later that imposing some properties on the set $F$ may
substantially lower the $\epsilon$-distortion complexity. 

In the case of dynamical systems, instead of looking at the $\epsilon$
distortion complexity of an attractor, it may be interesting to look
at the distortion complexity of a (invariant) probability measure.
Recall that for probability measures on a compact set, the
Kantorovich distance is a metric for the weak topology. We recall
that (see \cite{rachev}) it is defined by
$$
d_{K}\big(\mu,\nu\big)=
\inf_{f\in\mathscr{L}_{1}}\left(\int f\;d\mu- \int f\;d\nu\right)\;,
$$
where $\mathscr{L}_{1}$ is the set of Lipschitz continuous functions
with Lipschitz constant at most one.

We can now consider programs whose output 
are atomic measures with a finite number of atoms and rational
coefficients. In analogy with definition \ref{disto} we can define the 
$\epsilon$-distortion complexity of a measure.

\begin{definition}
For $\epsilon>0$,  the $\epsilon$-distortion complexity
of a measure $\mu$ (denoted by $\Delta_{\epsilon}(\mu)$)  is the number
$$
\Delta_{\epsilon}(\mu)=\min\big\{l(\programme)\,
\big|\,\programme\in\ensprog,\; 
d_{K}\big(\mu,\out(\programme)\big)\le \epsilon
\big\}
$$ 
\end{definition}

It is easy to verify that the $\epsilon$-distortion complexity of the
Lebesgue measure is bounded above by
$\mathcal{O}(1)\log\epsilon^{-1}$. Since a better precision does not
hurt, one can get a better bound by using a precision $1/n<\epsilon$
if the number $n$ has a lower complexity. 

Before we state our results, we recall the definition of Cantor sets
in the interval $[0,1]$ associated to iterated function systems. 

Let $\ind$ be a finite set of indices with at
least two elements. An (hyperbolic) \emph{Iterated Function System}
is a collection 
$$
\big\{\phi_i : A \to A  : i\in \ind\big\}
$$
of contractions on $A$ with uniform contraction rate (Lipschitz constant)
$\rho \in (0,1)$, and such that $\phi_i(A) \cap \phi_j(A) =
\emptyset$ for $i\not= j$. We shall only consider hyperbolic
iterated function systems with injective contractions (IHIFS for short).

If $\omega_{0}^{\infty}=(\omega_{0},\omega_{1},\ldots)$ 
is an infinite sequence of indices (an element of
$\ind^{\entier}$), the set 
$$
\cantor:= \bigcup_{\omega_{0}^{\infty} \in \ind^\entier}\
\bigcap_{n=0}^\infty \ \phi_{\omega_{n}}\cdots\phi_{\omega_{0}}([0,1])
$$
is a Cantor set and satisfies
\begin{equation} 
\cantor = \bigcup_{i\in \ind} \ \phi_i(\cantor).
\end{equation}
It contains all the accumulation points of the images of all the finite
composition products of $\phi_{i}$'s. We will use sometimes the
notation $\mathscr{C}(\phi)$ to emphasize the collection of maps used
to construct the Cantor set.

For $q\ge p$, we will denote by $\phi_{\omega_{p}^{q}}$ the map
$$
\phi_{\omega_{p}^{q}}=\phi_{\omega_{q}}\circ\cdots\circ\phi_{\omega_{p}}\;.
$$
If $p>q$ we define $\phi_{\omega_{p}^{q}}$ to be the identity.

The original Cantor set is obtained  by using $\ind=\{0,1\}$,
$\phi_{0}(x)=x/3$ and $\phi_{1}(x)=(2+x)/3$. It is easy to verify that
the $\epsilon$-distortion complexity of this Cantor set grows at most 
like $\log \epsilon^{-1}$. It was observed by
A.Mandel  that the  $\epsilon$-distortion complexity of this Cantor
set grows as a
function of $\epsilon^{-1}$ more slowly than any computable function.

We now state some of the results obtained in \cite{bcc} in the case of
Cantor sets. 

\begin{theorem} \label{thm:pol}
Let $\cantor$ be a Cantor set generated by an IHIFS with polynomial
functions. Then
$$
\Delta_{\epsilon}(\cantor) \le \mathcal{O}\big(\log(\epsilon^{-1}))\;.
$$
Moreover,  there exist (many) polynomial IHIFS's (with
$\mathrm{Card}\;\ind=2$) 
such that the generated Cantor set satisfies
$$
\Delta_{\epsilon}(\cantor)\ge \delta \big(\log(\epsilon^{-1}))\;,
$$
for some $\delta>0$
\end{theorem}

\begin{theorem} \label{thm:analytic}
Let $\cantor$ be a Cantor set generated by an IHIFS with real analytic
functions. Then
$$
\Delta_{\epsilon}(\cantor)\leq
\mathcal{O}\big(\big(\log(\epsilon^{-1})\big)^2\big)
$$
\end{theorem}
We will establish a  converse result in section \ref{secana}.

\begin{theorem}\label{thm:ckupper}
Let $k> 1$. For any $\delta>0$, for any $C^k$ Cantor set $\cantor$
generated by an IHIFS with
box counting dimension $D$, we have
$$
\Delta_{\epsilon}(\cantor)\leq \mathcal{O}\big(\epsilon^{-\frac D k -\delta}\big)
$$
Moreover, for any $\delta>0$, there exist (many) $C^{k}$ 
Cantor sets $\cantor$ generated by an IHIFS (with $\mathrm{Card}\;\ind=2$), 
with box counting dimension at most $D+\delta$, such that for any
$\epsilon>0$ small enough,
$$
\epsilon^{-\frac D k+\delta}
\leq \Delta_{\epsilon}(\cantor).
$$
\end{theorem}

The word ``many'' in the second parts of Theorems \ref{thm:pol} and
\ref{thm:ckupper} is given a precise (probabilistic) meaning in
\cite{bcc}. We will show in section \ref{secana} an analogous result
in the analytic case where this word means large cardinality (the
combinatorial complexity in the sense of Kolmogorov's paper \cite{K0}). 

Here we want to emphasis the difference of behavior between 
Cantor sets defined by analytic IHIFS and Cantor sets defined by
differentiable  IHIFS as appearing in Theorems \ref{thm:analytic} and 
\ref{thm:ckupper} respectively. In the first case the upper bound
grows rather  slowly while in the second case, the growth is much
faster and depends crucially on two properties of the system, the box
counting dimension of the set and the regularity of the IHIFS.

These results (and some others) are proved in details in \cite{bcc}. 
In the next sections we will explain some of the ideas behind these
proofs applied to two new results. In section \ref{secana} we will
prove a lower bound in the analytic case, and in section \ref{secdyn}
we will prove an upper bound for the $\epsilon$-distortion complexity
of the attractor of a uniformly hyperbolic dynamical system.

\section{The $\epsilon$-distortion complexity of real
analytic IHIFS.}\label{secana}

In this section we will obtain some results 
for Cantor sets of the interval defined by analytic IHIFS. The results
are formulated in terms of 
the so called combinatorial complexity (see
\cite{K0} and \cite{kt}), namely we will ``count'' the number of
Cantor sets with a given property. The relation with Kolmogorov
complexity is that since the number of programs of length $p$ is at
most $2^{p}$, if you have a set with larger cardinality, some  of its
elements should have complexity larger than $p$.

We start by introducing  some definitions and
notations.  
We choose once for all a  number $R>1/2$ and for any $0<\tilde
\rho<\rho<1/2$, $Q>1+R/2$,   we denote by $\mathcal{H}_{R}(\tilde\rho,\rho,Q)$ the
set of functions $f$ satisfying
\begin{enumerate}[i)]
\item $f$ is analytic and has modulus 
bounded by $Q$
 in the complex disk $D_{R}$ centered at $z=1/2$ and with radius
$R$.  
\item $f$ maps the interval $[0,1]$ into itself, 
and satisfies
$$
\tilde\rho\le \inf_{x\in[0,1]}f'(x)\le \sup_{x\in[0,1]}f'(x)\le \rho\;.
$$
\end{enumerate}
Note that this set is non empty since it contains the function
$f(x)=(\rho+\tilde\rho)(x-1/2)/2+1/2$. 
We will use the indices $0$ and $1$ to specify the components of the
elements of $\mathcal{H}_{R}(\tilde\rho,\rho,Q)^{2}$. We will denote by 
$\mathcal{K}_{R}(\tilde\rho,\rho,Q)$ the subset of
$\mathcal{H}_{R}(\tilde\rho,\rho,Q)^{2}$ given by
$$
\mathcal{K}_{R}(\tilde\rho,\rho,Q)=\big\{\phi=(\phi_{0},\phi_{1})
\in \mathcal{H}_{R}(\tilde\rho,\rho,Q)^{2}\,,\, \phi_{0}(0)=0,\; 
\phi_{1}(1)=1\big\}\;.
$$
Given two Cantor sets $\mathcal{C}$ and $\mathcal{C}'$, we will say
that they are $\epsilon$-separated if 
$$
d\big(\mathcal{C},\mathcal{C}'\big)>\epsilon\;.
$$
For $\epsilon>0$, we will denote by
$\mathscr{N}(\epsilon,\tilde\rho,\rho,Q,R)$ the 
maximal number of pairwise $\epsilon$-separated Cantor sets 
defined by analytic IHIFS with
$\mathrm{Card}\;\ind=2$ and $\phi\in \mathcal{K}_{R}(\tilde\rho,\rho,Q)$.
This is the $\epsilon/2$ capacity as defined in \cite{kt}.

Before we give the main result of this section 
which is new with
respect to \cite{bcc}, and complements Theorem \ref{thm:analytic}, we
prove some technical lemmas.

\begin{lemma}\label{bornecomp} 
For any $\phi^{0}\in \mathcal{K}_{R}(\tilde\rho,\rho,Q)$, 
  if $g_{0}$
and $g_{1}$ are analytic and bounded in $D_{R}$ and 
satisfy $g_{0}(0)=g_{1}(1)=0$, and
$$
\sup_{x\in[0,1]}|g_{0}'(x)|+\sup_{x\in[0,1]}|g_{1}'(x)|
<\inf\big\{\tilde\rho,1/2-\rho\big\}\,
$$
then $\phi=\big(\phi_{0}^{0}+g_{0},\phi_{1}^{0}+g_{1}\big)$ satisfies
$$
\sup_{x\in[0,1]}\sup_{\omega_{0}^{n}\in\indb^{n+1}}\big|\phi_{\omega_{0}^{n}}(x)
-\phi^{0}_{\omega_{0}^{n}}(x)\big|\le \frac{1}{1-\rho}\max\big\{
\sup_{x\in[0,1]}\big|g_{0}(x)\big|\,,\, 
\sup_{x\in[0,1]}\big|g_{1}(x)\big|\big\}
$$
\end{lemma}
\begin{proof} 
The proof is recursive. \qed
\end{proof}

For any integer $n$, we will denote by $\mathscr{C}_{n}(\phi)$ the set
$$
\mathscr{C}_{n}(\phi)=\bigcup_{\omega_{0}^{n-1}\in\indb^{n}}\phi_{\omega_{0}^{n-1}}(0)\;
\bigcup\; \bigcup_{\omega_{0}^{n-1}\in\indb^{n}}\phi_{\omega_{0}^{n-1}}(1)
$$
Note that this set has cardinality $2^{n}+2$.

\begin{lemma}\label{difference}
For any integer $n\ge2$ we have
$$
\mathscr{C}_{n}(\phi)\backslash \mathscr{C}_{n-1}(\phi)=
\bigcup_{\omega_{0}^{n-2}\in\indb^{n-1}}\phi_{\omega_{0}^{n-2}1}(0)\;
\bigcup\; \bigcup_{\omega_{0}^{n-2}\in\indb^{n-1}}\phi_{\omega_{0}^{n-2}0}(1)
$$
\end{lemma}
\begin{proof}
The proof is recursive.
\qed\end{proof}
\begin{lemma}\label{interpol} Let $(\xi_{u})_{u\in
    \mathscr{C}_{n}(\phi)}$ be a (finite) sequence of complex numbers.
Then for any $\phi\in \mathcal{K}_{R}(\tilde\rho,\rho,Q)$, the  polynomial
$$
h(z)=\sum_{u\in\mathscr{C}_{n}(\phi) }\xi_{u}\prod_{\stackrel{y\in
  \mathscr{C}_{n}(\phi)}{  y\neq u}}\frac{z-y}{u-y}
$$
satisfies
$$
\sup_{z\in D_{R}}|h(z)|+ \sup_{z\in D_{R}}|h'(z)|\le
16\;4^{n}\;\left(\frac{1}{2}+R\right)^{2^{n}+2}
\;\tilde\rho^{-8\,2^{n}}\sup_{z\in \mathscr{C}_{n}(\phi)}\big|\xi_{z}\big|\;. 
$$
\end{lemma}
\begin{proof}
If two points $u$ belongs to  $\mathscr{C}_{n}(\phi)$, the nearest
point in $\mathscr{C}_{n}(\phi)$ is at distance at least $\tilde
\rho^{n}$, the two next nearest neighbor are at distance at least
$\tilde \rho^{n-1}$ etc. The result follows. \qed
\end{proof}

\begin{lemma}\label{disjoints}
Let $\phi\in \mathcal{K}_{R}(\tilde\rho,\rho,Q)$. 
For any integer $n$, when $\omega_{0}^{n-1}$  varies in $\{0,1\}^{n}$,
the intervals
$$
[\phi_{\omega_{0}^{n-1}}(0),\phi_{\omega_{0}^{n-1}}(1)]
$$ 
are pairwise disjoint, and their pairwise distance is at least
$\tilde \rho^{n}(1-2\rho)$. Their union contains the Cantor set 
$\mathscr{C}(\phi)$. Moreover, the points in  $\mathscr{C}_{n}(\phi)$
belong to $\mathscr{C}(\phi)$.
\end{lemma}
\begin{proof}
The proof follows easily recursively 
from the fact that $\rho<1/2$, and the
point $0$ (respectively $1$) is fixed by the map $\phi_{0}$
(respectively $\phi_{1}$).
\qed\end{proof}

The following Lemma is a particular version of Lemma 4.1 in \cite{bcc}
($I=I'=[0,1]$ in the notations of this paper),
and will be used to get  a lower bound on the Hausdorff distance of
two closed sets $F$ and $F'$ with holes $H$ and $H'$. 

\begin{lemma}\label{infhaus}
Let $F$ and $F'$ be two closed subsets of $[0,1]$. Let $H=[c,d]$ and
$H'=[c',d']$ be closed sub-intervals of $]0,1[$. Assume that
$$
\{c,d\}\subset F,\; \{c',d'\}\subset F',\;  
F\cap \stackrel{\scriptscriptstyle{\circ}}{H} =
\emptyset,\; F'\cap \stackrel{\scriptscriptstyle{\circ}}{H'} =
\emptyset\;.
$$
 Moreover, assume that  for some $\epsilon>0$
$$
|c-d|>2\epsilon,\; |c'-d'|>2\epsilon,\; \max
\big\{|c-c'|,|d-d'|\big\}>\epsilon\;.
$$
Then
$$
d(F,F')>\epsilon\;.
$$
\end{lemma}
The proof follows at once from the definition of the Hausdorff
distance.

We see from this result and Lemma \ref{disjoints} that in order to
construct ``many'' Cantor sets it is enough to construct the points
$\phi_{\omega_{0}^{k-1}}(0)$ and $\phi_{\omega_{0}^{k-1}}(1)$. However
since we want these points to be generated by analytic maps there are
some constraints.

The following result provides an upper bound for $\mathscr{N}(\epsilon)$

\begin{theorem}
For any $R$, $\rho$ and $\tilde\rho$ and $Q$ as above, 
there exists a constant $K=K(R,\rho\tilde,\rho,Q)>0$ such that for any
$1/2>\epsilon>0$ 
$$
\mathscr{N}(\epsilon)\le K\;\big(\log\epsilon\big)^{2}\:.
$$
\end{theorem}
\begin{proof} 
The proof is analogous to the proof of Theorem \ref{thm:analytic} (see
\cite{bcc}) but we sketch it for the convenience of the reader.
Let  $\phi\in  \mathcal{K}_{R}(\tilde\rho,\rho,Q)$, we have for any
$z$ in the interior of $D_{R}$
$$
\phi_{0}(z)=\sum_{n=0}^{\infty} \phi_{0}^{(n)}\left(z-\frac{1}{2}\right)^{n}
$$
with
$$
\big| \phi_{0}^{(n)}\big|\le \frac{Q}{R^{n}}\;.
$$
A similar expression and estimate hold for $\phi_{1}$. The number
$\phi_{0}^{(n)}$ is $n!$ times the $n^{\mathrm{th}}$ derivative of
$\phi_{0}$ in $z=1/2$.  Let
$$
N=\left\lceil\frac{\log\big(16\,Q/[(2R-1)\epsilon(1-\rho)]}{\log(2R)}\;.
\right\rceil
$$
and define
$$
\tilde \phi_{0}(z)=\sum_{n=0}^{N}
\phi_{0}^{(n)}\left(z-\frac{1}{2}\right)^{n},\qquad
\tilde \phi_{1}(z)=\sum_{n=0}^{N}
\phi_{1}^{(n)}\left(z-\frac{1}{2}\right)^{n}.
$$
We have
$$
\sup_{x\in[0,1]}\big|\phi_{0}(x)-\tilde\phi_{0}(x)\big| \le 
\sum_{n=N+1}^{\infty}\frac{Q}{(2R)^{n}}=\frac{Q}{(2R)^{N}(2R-1)}
\le \frac{\epsilon(1-\rho)}{16}\;.
$$
We define for each  $0\le n\le N$ and $\epsilon>0$ the finite set
$$
\mathscr{B}_{\epsilon,n}=\left\{
(p+iq)\;\epsilon \;(1-\rho)/16\,,\, -\frac{17\,Q}{(2R)^{n}\epsilon(1-\rho)} 
\le p,q\le \frac{17\,Q}{(2R)^{n}\epsilon(1-\rho)} \right\}\;.
$$
Let $\mathscr{Q}_{\epsilon,N}$ denote the set of  finite sequences of
complex numbers given by 
$$
\mathscr{Q}_{\epsilon,N}=\big\{(\xi_{r})_{0\le r\le N}\;,\;
\xi_{r}\in \mathscr{B}_{\epsilon,r}\; \forall \;0\le r\le N\big\}
$$
To each finite sequence of complex numbers $\underline\xi$
in $\mathscr{Q}_{\epsilon,N}$,
we associate the polynomial 
$$
f_{\underline\xi}(z)=\sum_{r=0}^{N}\xi_{r}\left(z-\frac{1}{2}\right)^{r}.
$$
It is easy to verify that for any  $\phi\in
\mathcal{K}_{R}(\tilde\rho,\rho,Q)$, there is a sequence 
$\underline\xi\in \mathscr{D}_{\epsilon,N}$  such that
$$
\sup_{x\in[0,1]}\big|f_{\underline\xi}(x)
-\phi_{0}(x)\big|<\frac{(1-\rho)\epsilon}{8}\;.
$$
A similar estimate holds for $\phi_{1}$ (with in general another
sequence  $\underline\xi$).
This implies that we can find a collection  $\mathscr{G}_{\epsilon}$ of
elements of $\mathcal{K}_{R}(\tilde\rho,\rho,Q)$ with cardinality at
most 
$$
\mathrm{Card}\big(\mathscr{G}_{\epsilon}\big)\le 
\mathrm{Card}\big(\mathscr{Q}_{\epsilon,N}\big)^{2}\le 
\left(\frac{34\,Q}{\epsilon(1-\rho)}\right)^{N}\;,
$$
such that for any $\phi\in\mathcal{K}_{R}(\tilde\rho,\rho,Q)$, we can
find a $\tilde \phi\in \mathscr{G}_{\epsilon}$ such that
$$
\max\big\{\sup_{x\in[0,1]}\big|\phi_{0}(x)-\tilde\phi_{0}(x)\big|,
\sup_{x\in[0,1]}\big|\phi_{1}(x)-\tilde\phi_{1}(x)\big|\big\}\le 
\frac{(1-\rho)\epsilon}{4}\;.
$$
This also follows directly from results in \cite{kt}.

From Lemmas \ref{bornecomp} and \ref{disjoints}, we have 
$$
d\big(\mathscr{C}(\phi),\mathscr{C}(\tilde\phi))\le \epsilon\;.
$$
This finishes the proof of the Theorem.
\qed\end{proof}

\begin{proposition}\label{propit}
 Given $R$, $\tilde\rho$, $\rho$ and $Q$ as above,
  there exists 
$\tilde\epsilon=\epsilon(R,\tilde\rho,\rho,Q)>0$ 
such that if 
$\phi^{0}\in
\mathcal{K}_{R}(\tilde\rho,\rho,Q)$,
$0<\epsilon_{0}<\tilde\epsilon$, and
$$
0<\epsilon_{1}<\epsilon_{0}^{5+2\log_{2}(R+1/2)-16\log_{2}\tilde\rho}\;,
$$ 
there exists a subset
$\mathscr{Q}=\mathscr{Q}(R,\tilde\rho,\rho,Q,
\phi^{0},\epsilon_{0},\epsilon_{1})$ of 
$\mathcal{K}_{R}(\tilde\rho-\epsilon_{0},\rho+\epsilon_{0}
,Q+\epsilon_{0})$
with cardinality at least 
$e^{2^{-4}\log(\epsilon_{0}/\epsilon_{1})\log\epsilon_{0}^{-1}}$   such that for any $\phi\in
\mathscr{Q}$
$$
d\big(\mathscr{C}(\phi^{0}),\mathscr{C}(\phi)\big)\le
\frac{\epsilon_{0}}{3}\;, 
$$
and
$$
\sup_{z\in D_{R}}\big|\phi^{0}_{0}(z)-\phi_{0}(z)\big|+
\sup_{z\in D_{R}}\big|\phi^{0}_{1}(z)-\phi_{1}(z)\big|\le
 \frac{\epsilon_{0}}{4}\;.
$$
Moreover,  for any $\phi\neq \phi'$ in $\mathscr{Q}$
$$
d\big(\mathscr{C}(\phi),\mathscr{C}(\phi')\big)\ge \epsilon_{1}\;.
$$
\end{proposition}

\begin{proof}
For  $\epsilon_{0}>0$ we define the integer $N=N(\epsilon_{0}^{-1})$ by
$$
N=\big\lfloor \log_{2}\log(\epsilon_{0}^{-1})
\big\rfloor\;.
$$ 
In other words
$$
\frac{1}{2} \log(\epsilon_{0}^{-1})\le 
2^{N}\le \log(\epsilon_{0}^{-1})\;.
$$
We will denote by $\mathscr{J}_{N}$ the set of maps from 
$\{0,1\}^{N-2}$ to
$\big\{1,2,\ldots,\lfloor \sqrt{\epsilon_{0}/\epsilon_{1}}
\rfloor\big\}^{4}$. 

For an element $J\in \mathscr{J}_{N}$, we define 
a pair of functions $(g_{0}^{J},g_{1}^{J})$ as follows.
We start  by
defining the functions on the set $\mathscr{C}_{N-1}(\phi^{0})$ by
setting 
$$
g^{J}_{0}(x)=g^{J}_{1}(x)=0
$$
for any $x\in \mathscr{C}_{N-1}(\phi^{0})$. 

We next define the  the two functions on 
$\mathscr{C}_{N}(\phi^{0})\backslash  \mathscr{C}_{N-1}(\phi^{0})$.
Using Lemma \ref{difference}, this is done by setting for any
$\omega_{0}^{N-2}\in \indb^{N-1}$
$$
g^{J}_{0}(\phi^{0}_{\omega_{0}^{N-2}1}(0))=
-J_{1}(\omega_{0}^{N-2})\;\epsilon_{1}\;,\quad
g^{J}_{0}(\phi^{0}_{\omega_{0}^{N-2}0}(1))=
J_{2}(\omega_{0}^{N-2})\;\epsilon_{1}\;,
$$
$$
g^{J}_{1}(\phi^{0}_{\omega_{0}^{N-2}1}(0))=
-J_{3}(\omega_{0}^{N-2})\;\epsilon_{1}\;,\quad
g^{J}_{1}(\phi^{0}_{\omega_{0}^{N-2}0}(1))=
J_{4}(\omega_{0}^{N-2})\;\epsilon_{1}\;.
$$
Finally, we define the functions as polynomials by the Lagrange
interpolation formula 
$$
g^{J}_{0}(x)=\sum_{z\in
\mathscr{C}_{N}(\phi^{0})} g^{J}_{0}(z)\prod_{\stackrel{y\in
\mathscr{C}_{N}(\phi^{0})}{y\neq z }}\frac{x-y}{z-y}\;,
$$
and similarly for $g^{J}_{1}$. 

From the definition of $N$ and the condition on $\epsilon_{1}$ we get
(for $\tilde\epsilon$ small enough)
from  Lemma \ref{interpol} and for 
any  $J\in \mathscr{J}_{N}$  
$$
\sup_{z\in D_{R}}|g^{J}_{0}(z)|+\sup_{z\in D_{R}}|{g^{J}}'_{0}(z)|+
\sup_{z\in D_{R}}|g^{J}_{1}(z)|+\sup_{z\in D_{R}}|{g^{J}}'_{1}(z)|
$$
$$
\le 4
2^{N}\left(\frac{1}{2}+R\right)^{2^{N}} \tilde\rho^{-8\,2^{N}}
\rho^{N}\;\sqrt{\epsilon_{0}\,\epsilon_{1}}\le \frac{\epsilon_{0}}{4}\;,
$$
which implies
$\big(\phi_{0}^{0}+g_{0}^{J},\phi_{1}^{0}+g_{1}^{J}\big)
\in\mathcal{K}_{R}(\tilde\rho-\epsilon_{0},\rho+\epsilon_{0},
Q+\epsilon_{0})$. 

Using Lemmas \ref{disjoints}, \ref{infhaus}, 
and $\tilde\rho^{N}>\epsilon_{1}(1-2\rho)$, 
we conclude that if 
$J$ and $J'$ are two different elements of $\mathscr{J}_{N}$, we have 
$$
d\big(\mathscr{C}(\phi_{0}^{0}+g_{0}^{J},\phi_{1}^{0}+g_{1}^{J}),
\mathscr{C}(\phi_{0}^{0}+g_{0}^{J'},\phi_{1}^{0}+g_{1}^{J'})\big)
>\epsilon_{1}\;.
$$

Using Lemmas \ref{disjoints} and  \ref{bornecomp},
 we conclude that (for $\epsilon_{0}$ small enough), for any $J\in
\mathscr{J}_{N}$
$$
d\big(\mathscr{C}(\phi^{0}),
\mathscr{C}(\phi_{0}^{0}+g_{0}^{J},\phi_{1}^{0}+g_{1}^{J})\big)\le
\frac{\sqrt{\epsilon_{0}\,\epsilon_{1}}}{1-\rho}<
\frac{\epsilon_{0}}{3}\;.
$$
The proposition follows from the estimate
$$
\mathrm{Card}(\mathscr{J}_{N})=
\left\lfloor\sqrt{\frac{\epsilon_{0}}{\epsilon_{1}}}\right\rfloor^{2^{N-2}}\;.
$$
\qed\end{proof}

We can now state and prove the main result of this section which
complements Theorem \ref{thm:analytic}.

\begin{theorem}\label{laborneinf}
There exists Cantor sets $\mathcal{C}$ defined by analytic IHIFS 
such that
$$
\liminf_{\epsilon\to
  0}\frac{\Delta_{\epsilon}(\mathcal{C})}{(\log\epsilon^{-1})^{2}} >0\;.
$$
\end{theorem}
\begin{proof}
Let $R$, $\tilde\rho_{0}$, $\rho_{0}$ and $Q_{0}$ satisfy the
assumptions defined at the beginning of this section.
We will use  the sequence of numbers 
$$
\epsilon_{p}=2^{-K\;\gamma^{p}}\;,
$$
with $K>0$ (large enough) and $\gamma>1$ to be chosen later on. Let 
$$
\rho_{p+1}=\rho_{p}+\epsilon_{p},\;
\tilde \rho_{p+1}=\tilde\rho_{p}-\epsilon_{p},\;
Q_{p+1}=Q_{p}+\epsilon_{p}\;.
$$
We assume that $K$ is large enough so that
$$
\frac{\tilde\rho_{0}}{2}
<\inf_{p} \tilde\rho_{p}\le\sup_{p} \rho_{p}<\frac{2\,\rho_{0}+1}{4}\;.
$$
Let $V>0$ and $\Gamma>0$ to be chosen later on.
We consider the recursive assumption:

For any integer $p\ge0$, 
there exists a subset $\mathscr{E}_{p}$ of
$\mathcal{K}_{R}(\tilde\rho_{p},\rho_{p},Q_{p})$, with the following
properties. 
\begin{enumerate}[1)]
\item The  cardinality of $\mathscr{E}_{p}$ is at least
$$
2^{V(\log_{2}\epsilon_{p})^{2}}\;.
$$
\item For any $0\le \ell\le p$
$$
\inf_{\phi\in \mathscr{E}_{p}}\Delta_{\epsilon_{\ell}}\big(\mathscr{C}(\phi)\big)
\ge \Gamma (\log_{2}\epsilon_{\ell})^{2}\;.
$$
\item For any $\phi\neq \phi'\in \mathscr{E}_{p}$,
$$
d\big(\mathscr{C}(\phi),\mathscr{C}(\phi')\big)>\epsilon_{p}\;.
$$
\item If $p>0$, for any $\phi\in \mathscr{E}_{p}$, there exists
$\tilde\phi\in \mathscr{E}_{p-1}$ such that
$$
d\big(\mathscr{C}(\tilde\phi),\mathscr{C}(\phi)\big)\le
\frac{\epsilon_{p-1}}{3}\;, 
$$
and
$$
\sup_{z\in D_{R}}\big|\tilde\phi_{0}(z)-\phi_{0}(z)\big|+
\sup_{z\in D_{R}}\big|\tilde\phi_{1}(z)-\phi_{1}(z)\big|\le \epsilon_{p-1}\;.
$$
\end{enumerate}

For $p=0$, we can find  $\Gamma>0$ and  $V>0$ depending on $K$ such
that for any $K>0$ large enough, the above assumption is satisfied.  

We now prove that we can choose the constants $K$, $\gamma$, $\Gamma$,
$V$, such that if the assumption is true at step $p\ge0$, it will
also be true at step $p+1$.

We impose the lower bound
$$
\gamma>5+2\log_{2}(R+1/2)-16\log_{2}(\tilde\rho_{0}/2)\;.
$$
This implies for any $p\ge0$
$$
\gamma>5+2\log_{2}(R+1/2)-16\log_{2}(\tilde\rho_{p})\;.
$$

To each element  $\phi^{0}\in\mathscr{E}_{p}$, we apply Proposition
\ref{propit} with $\epsilon_{0}$ replaced by $\epsilon_{p}$ and 
$\epsilon_{1}$ replaced by $\epsilon_{p+1}$.

We obtain a finite subset $\mathscr{X}_{p}$ of
$\mathcal{K}_{R}(\tilde\rho_{p+1},\rho_{p+1},Q_{p+1})$
with cardinality
$$
2^{2^{-4}(\log\epsilon_{p+1}-\log\epsilon_{p})\log\epsilon_{p}
+ V(\log_{2}\epsilon_{p})^{2}}\;.
$$ 
Assume $\phi\neq\phi'\in  \mathscr{X}_{p}$ come from the
application of  Proposition \ref{propit}  to the same
$\phi^{0}\in\mathscr{E}_{p}$, then
$$
d\big(\mathscr{C}(\phi),\mathscr{C}(\phi')\big)>\epsilon_{p+1}\;.
$$ 
If they come from $\phi^{0}\in\mathscr{E}_{p}$ and $\tilde
\phi^{0}\in\mathscr{E}_{p}$ respectively with $\phi^{0}\neq \tilde 
\phi^{0}$, we have from the triangle inequality and the recursive
assumption (and $K$ large enough)
$$
d\big(\mathscr{C}(\phi),\mathscr{C}(\phi')\big)
$$
$$
\ge
d\big(\mathscr{C}(\phi^{0}),\mathscr{C}(\tilde\phi^{0})\big)
-d\big(\mathscr{C}(\phi),\mathscr{C}(\phi^{0})\big)
-d\big(\mathscr{C}(\phi'),\mathscr{C}(\tilde\phi^{0})\big)
$$
$$
\ge \epsilon_{p}-2\frac{\epsilon_{p}}{3}\ge \epsilon_{p+1}\;.
$$
Therefore, for each $\phi\neq\phi'\in  \mathscr{X}_{p}$ we have
$$
d\big(\mathscr{C}(\phi),\mathscr{C}(\phi')\big)>\epsilon_{p+1}\;,
$$
which is the third part of the recurrence assumption. 

Since the number of programs of length $\Gamma
(\log_{2}\epsilon)^{2}$ is at most $2^{\Gamma
  (\log_{2}\epsilon)^{2}}$, there is  a subset $\mathscr{E}_{p+1}$ of 
$\mathscr{X}_{p}$ with cardinality at least 
$$
2^{2^{-8}(\log\epsilon_{p+1}-\log\epsilon_{p})\log\epsilon_{p}
+ V(\log_{2}\epsilon_{p})^{2}}-\sum_{\ell=0}^{p+1}
2^{\Gamma (\log_{2}\epsilon_{\ell})^{2}}
$$
$$
\ge 
2^{2^{-8}(\log\epsilon_{p+1}-\log\epsilon_{p})\log\epsilon_{p}
+ V(\log_{2}\epsilon_{p})^{2}}-(p+2) \;2^{\Gamma (\log_{2}\epsilon_{p+1})^{2}}
$$
such that for all $0\le j\le p+1$
$$
\Delta_{\epsilon_{j}}\big(\mathscr{C}(\phi)\big)>
\Gamma (\log_{2}\epsilon_{j})^{2} \;.
$$
We now choose the numbers $\gamma$, $\Gamma$ and $V$ such that for any $p\ge0$
$$
2^{2^{-8}(\log\epsilon_{p+1}-\log\epsilon_{p})\log\epsilon_{p}
+ V(\log_{2}\epsilon_{p})^{2}}> 2^{V(\log_{2}\epsilon_{p+1})^{2}}+
(p+2) \;2^{\Gamma (\log_{2}\epsilon_{p+1})^{2}}\;.
$$
This will prove the first and second parts of the recursive assumption.
Observing that $\epsilon_{p+1}=\epsilon_{p}^{\gamma}$,
it is enough to have the above inequality 
to ensure that together with $K$ large enough, we have
$$
2^{-8}(\gamma-1)(\log\epsilon_{p})^{2}>
2\;\gamma^{2}\big(V+\Gamma)(\log\epsilon_{p})^{2}\;.
$$
For any given $\gamma>1$ this can be satisfied by taking for example
$$
0<V=\Gamma<\frac{2^{-8}(\gamma-1)}{5\;\gamma^{2}}\;.
$$
The last part of the recursive assumption follows directly from
Proposition \ref{propit}.  

It is easy to verify that the sequence of sets
$\mathcal{K}_{R}(\tilde\rho_{p},\rho_{p},Q_{p})$ is
increasing. Moreover, these are closed  sets for the sup
norm on analytic functions in $D_{R}$, and compact for the sup
norm on analytic functions  in $D_{R'}$ for any
$R'<R$ (by Montel's Theorem). 

Let $\mathscr{E}$ be the set of accumulation points of sequences in 
$(\mathscr{E}_{p})$. It is easy to verify from the last part of the
recursion assumption, that for $\gamma$ large enough, 
if $\phi\in \mathscr{E}$, for any $p\ge0$,
there exists $\phi^{(p)}\in \mathscr{E}_{p}$ such that
$$
d\big(\mathscr{C}(\phi^{(p)}),\mathscr{C}(\phi)\big)\le
\frac{1}{3}
\sum_{\ell=p}^{\infty}\epsilon_{\ell}\le \frac{\epsilon_{p}}{2}\;.
$$
Let $\phi\in \mathscr{E}$,  and $0<\epsilon<\epsilon_{2}$, there is a
unique $p$ such that 
$$
\frac{\epsilon_{p+1}}{2}<\epsilon\le\frac{\epsilon_{p}}{2}\;.
$$
Let $\tilde\phi\in \mathscr{E}_{p}$ be such that
$$
d\big(\mathscr{C}(\tilde\phi),\mathscr{C}(\phi)\big)\le 
\frac{\epsilon_{p}}{2} 
$$
It follows at once from the definition of the $\epsilon$-distortion
complexity that
$$
\Delta_{\epsilon_{p}/2}\big(\mathcal{C}(\phi)\big)\ge 
\Delta_{\epsilon_{p}}\big(\mathcal{C}(\tilde\phi)\big)\ge
\Gamma (\log\epsilon_{p})^{2}\ge \Gamma \; \gamma^{-2}\;
 (\log\epsilon)^{2}\;.
$$
Since this holds for any $\epsilon>0$, the theorem is proved. 
\qed\end{proof}

We observe that from the proof of the previous Theorem one can derive
a lower bound on the $\epsilon$-entropy (see \cite{kt}, or in other
words a lower bound on the combinatorial complexity (see \cite{K0}),
namely how many Cantor sets generated by analytic IHIFS are there which
differ at precision $\epsilon$.

\section{An upper bound for dynamical systems.}\label{secdyn}

In this section we prove  an  upper bound in the case of uniformly hyperbolic
diffeomeorphisms of compact subsets of $\real^{d}$.
\begin{theorem}\label{thm:sysdyn}
Let $f$ be a $C^{k}$ ($k\ge 1$) diffeomorphism of a compact subset $M$
of $\real^{d}$ ($d\ge 2$).
Let $\mathcal{A}$ denote an attractor of $f$. Assume  that on a
neighborhood $\mathscr{V}$ of $\mathcal{A}$ the diffeomorphism $f$ is
uniformly hyperbolic (see \cite{gh}) and denote by 
$\lambda_{-}<0<\lambda_{+}$ two numbers such that 
$$
\sup_{x\in \mathscr{V}}\big\|D_{x}f_{}\big\|\le e^{\lambda_{+}}\;,
$$
and for any $y\in \mathscr{V}$, 
$$
d\big(f(y),\mathscr{A})\le e^{\lambda_{-}}d\big(y,\mathscr{A})\;.
$$
Then for $\epsilon$ small enough, 
the $\epsilon$-distortion complexity of $\mathcal{A}$  is bounded
by
$$
\Delta_{\epsilon}(\mathcal{A})\le 
\mathcal{O}(1)\; \epsilon^{-d_{\mathcal{A}}
  (\lambda_{+}-\lambda_{-}/k)/(\lambda_{+}-\lambda_{-}) }\;
\big(\log\epsilon^{-1}\big)^{2}\;,
$$
where $d_{\mathcal{A}}$ denotes the box counting dimension of $\mathcal{A}$.
\end{theorem}
Note that if we take $\lambda_{-}=0$ we recover the trivial upper
bound \eqref{bornetriviale}, and similarly for $k=1$. 
 The number $\lambda_{-}<0$ can be easily related to the linear expansion
and contraction factors of the unstable and stable linear bundles (see
\cite{gh}).  
\begin{proof}
We give a sketch of the proof. Let $\delta>0$ to be chosen later on. 
For $\epsilon>0$ small enough, we need of the order of
$\epsilon^{-\delta d_{\mathcal{A}}}$ balls of radius $\epsilon^{\delta}$ to 
cover the attractor $\mathcal{A}$. We make a choice of such a covering
and denote by $\mathcal{C}$ the sets of centers of the balls.

Let $m$ be an integer to be chosen later on. For any point $x$ we have
$$
f^{m}(x+y)=f^{m}(x)+Df^{m}(x)(y)+D^{2}f^{m}(x)(y,y)/2+\ldots
$$
$$
+
D^{k-1}f^{m}(x)(y,\ldots,y)/(k-1)!+\mathcal{O}\big(\|y\|^{k}\;\big\|
D^{k}f^{m}\big\| \big)\;.
$$
By the chain rule $\big\|D^{k}f^{m}\big\|$ is at most of order
  $m^{k}\exp(m\,k\,\lambda_{+})$. We   impose the two conditions 
$$
e^{m\,k\,\lambda_{+}}\epsilon^{k\delta}<\epsilon\;,\quad\mathrm{and}\quad
e^{m\,\lambda_{-}}\epsilon^{\delta}<\epsilon\;.
$$
Eliminating $m$ between the two conditions, one gets
$$
\delta\ge\frac{\lambda_{+}-\lambda_{-}/k}{
\lambda_{+}-\lambda_{-}} \;,
$$
and we will use this minimal value of $\delta$. Note that we can take
$m=\mathcal{O}(\log\epsilon^{-1})$. 

We need to describe for each $x\in\mathcal{C}$ the quantities 
 $f^{m}(x)$, $Df^{m}(x)$ etc. up to
$D^{k-1}f^{m}(x)$,  at a precision of order $\epsilon$.
This gives a complexity bounded above by 
$\Oun \,m\,\log(exp(m\,k\,\lambda_{+}) \epsilon^{-1}m^{k})$.

Finally, in order to describe the attractor, we can start with a
regular lattice $\mathscr{L}_{m,\epsilon}$ with lattice size
$\epsilon\exp(-m\,\lambda_{+})$. Describing this lattice requires only
a complexity of at most $-\Oun \log(\epsilon\exp(-m\,\lambda_{+}))$ since
this is a regular lattice. 

For each $x\in \mathcal{C}$ we consider the set  
$f^{m}(\mathscr{L}_{m,\epsilon}\cap B_{\epsilon^{\delta}}(x))$.  All
the points in this set are within distance $\epsilon$ of $\mathcal{A}$
since $\epsilon^{\delta}e^{m\lambda_{-}}\le \epsilon$.

Moreover, for any $z\in \mathcal{A}$, $f^{-m}(z)\in \mathcal{A}$  and
therefore there is an $x\in \mathcal{C}$ such that
$$
d\big(f^{-m}(z),x\big)\le\epsilon^{\delta}\;.
$$
Therefore, there exists a point $\tilde y\in
\mathscr{L}_{m,\epsilon}\cap B_{\epsilon^{\delta}}(x)$ such that
$$
d\big(f^{-m}(z),\tilde y\big)\le \epsilon\exp(-m\,\lambda_{+})
$$
implying 
$$
d\big(z,f^{m}(\tilde y)\big)\le \epsilon\;.
$$
Therefore
$$
d\left(\mathcal{A},\cup_{x\in\mathcal{C}}
\big(\mathscr{L}_{m,\epsilon}\cap
B_{\epsilon^{\delta}}(x)\big)\right)\le \epsilon\;,
$$ 
and the result follows.
\qed\end{proof}

\section{Remarks and open questions.}

In this section we state some open problems which naturally arise from
the previous results. 

\subsection{Some questions about Cantor sets.}

In the proof of Theorem \ref{laborneinf}, in order to prove that there
exists a large enough collection of Cantor sets satisfying the lower
bound, we constructed many polynomials. These polynomials are of
course analytic and even entire functions, but the whole collection
cannot be considered from the point of view of Theorem \ref{thm:pol}
because their degree depends on $\epsilon$ (it is of order $\log
\epsilon^{-1}$). This raises the question of understanding better
this construction. One can try to use instead of the Lagrange
interpolation formula the Carleson interpolation formula (see \cite{Garnett}
for example). I expect this may improve the constant in front of the
$(\log \epsilon^{-1})^{2}$ but a more interesting question would be
to understand where  data like the dimension appear in the asymptotic
behavior of the $\epsilon$-distortion complexity when $\epsilon$
tends to zero (in the prefactor?).

In another direction, one may ask how to fill the gap between the
$\log\epsilon^{-1}$ in Theorem  \ref{thm:pol} and the $(\log
\epsilon^{-1})^{2}$ in Theorems \ref{thm:analytic} and
\ref{laborneinf}. A natural candidate would be to look at entire
functions of various order (see for example
\cite{Levin} for definition and results). Between the $(\log
\epsilon^{-1})^{2}$ behavior for analytic functions and the behavior
as a power of $\epsilon^{-1}$ for $C^{k}$ functions (Theorems
\ref{thm:analytic}, \ref{laborneinf} and \ref{thm:ckupper}),
  one can try to
fill the gap by looking at quasi-analytic functions (see for example
\cite{Levin} for definition and results).

The proof of Theorem \ref{thm:ckupper} in \cite{bcc} was based on the
use of the scaling function for fractal sets. Is it possible to give a
proof based on direct interpolation as we did for Theorem \ref{laborneinf}?

Some of the results in \cite{bcc} can probably be extended to Cantor
sets in higher dimension. The question of the complexity of measures
is essentially untouched as far as I know. Is it related to other
quantities like dimension and capacity? (see for example
\cite{kahanerand} for definitions and properties).

Note also that it is not clear if we can get generic (or prevalent)
results (see \cite{kahane} for definitions), these would have to be
formulated in the Hausdorff metric of the set of Cantor sets. 
We do  not expect  to have such properties from the point of view of
the set of IHIFS. More precisely, 
 it seems possible for example that the set of
real analytic IHIFS leading to a Cantor set with  
$\epsilon$-distortion complexity bounded above by $\Oun (\log
\epsilon^{-1})^{2-\sigma}$ for some $\sigma>0$ for any $\epsilon$
small enough is of second category.  

\subsection{Some questions about dynamical systems.}

The estimate in Theorem \ref{thm:sysdyn} can be easily extended to Riemannian
manifolds, but is of course very rough and one would like to 
use Lyapunov exponents instead of uniform bounds. 
The result should then involve invariant
measures. We formulate a conjecture in this direction.

\noindent\textsl{Conjecture. 
Let $\mu$ be an SRB measure for a $C^{k}$ diffeomorphism $f$ of a
compact surface with Lyapunov exponents $\lambda_{-}<0<\lambda_{+}$,
and dimension $d_{\mu}$. 
Then  
$$
\limsup_{\epsilon\to 0} \frac{ \log 
\Delta_{\epsilon}(\mu)}{\log \epsilon^{-1}}\le 
d_{\mathcal{\mu}}\;
 \frac{\log\lambda_{+}-\frac{\displaystyle
     \log\lambda_{-}}{\displaystyle k}}{
\log\lambda_{+}-\log\lambda_{-} }\;.
$$
}
We also conjecture that the above bound is saturated for many
diffeomorphims (in an adequate sense).


\end{document}